\documentclass[11pt,a4paper]{amsart}
\usepackage[T1]{fontenc}
\usepackage{lmodern}
\usepackage[margin=29mm]{geometry}
\usepackage{amsmath,amssymb,mathtools}
\usepackage{dsfont}
\usepackage{mathrsfs}
\usepackage{microtype}
\usepackage{xcolor}
\usepackage[colorlinks=true,linkcolor=blue!45!black,citecolor=blue!45!black,urlcolor=blue!45!black]{hyperref}
\hypersetup{pdftitle={A minimal type of Morley rank omega in a partial differential field},pdfsubject={A Burgers-equation type with U-rank one and Morley rank omega},pdfkeywords={DCF02, minimal types, Morley rank, Burgers equation}}
\usepackage{enumitem}
\setlist[enumerate]{label=\textup{(\roman*)},leftmargin=*,itemsep=3pt}
\allowdisplaybreaks[1]
\newtheorem{theorem}{Theorem}[section]
\newtheorem{proposition}[theorem]{Proposition}
\newtheorem{lemma}[theorem]{Lemma}
\newtheorem{corollary}[theorem]{Corollary}
\theoremstyle{definition}

\newtheorem{remark}[theorem]{Remark}
\newtheorem{question}[theorem]{Question}

\newcommand{\DCF}{\operatorname{DCF}}
\newcommand{\RM}{\operatorname{RM}}
\newcommand{\MD}{\operatorname{MD}}
\newcommand{\acl}{\operatorname{acl}}

\newcommand{\tp}{\operatorname{tp}}
\newcommand{\trdeg}{\operatorname{trdeg}}
\newcommand{\ord}{\operatorname{ord}}

\newcommand{\Ucal}{{\mathds{U}}}
\newcommand{\C}{\mathcal C}

\newcommand{\Uu}{{\mathds{U}}}

\newcommand{\dd}{\mathrm d}

\makeatletter
\def\moverlay{\mathpalette\mov@rlay}
\def\mov@rlay#1#2{\leavevmode\vtop{%
   \baselineskip\z@skip \lineskiplimit-\maxdimen
   \ialign{\hfil$\m@th#1##$\hfil\cr#2\crcr}}}
\newcommand{\charfusion}[3][\mathord]{
    #1{\ifx#1\mathop\vphantom{#2}\fi
        \mathpalette\mov@rlay{#2\cr#3}
      }
    \ifx#1\mathop\expandafter\displaylimits\fi}
\makeatother

\def\Ind#1#2{#1\setbox0=\hbox{$#1x$}\kern\wd0\hbox to
  0pt{\hss$#1\mid$\hss} \lower.9\ht0\hbox to
  0pt{\hss$#1\smile$\hss}\kern\wd0}
\def\Notind#1#2{#1\setbox0=\hbox{$#1x$}\kern\wd0\hbox to
  0pt{\mathchardef\nn="0236\hss$#1\nn$\kern1.4\wd0\hss}\hbox to
  0pt{\hss$#1\mid$\hss}\lower.9\ht0 \hbox to
  0pt{\hss$#1\smile$\hss}\kern\wd0}
\def\ind{\mathop{\mathpalette\Ind{}}}

\title[A minimal type of Morley rank $\omega$]{A minimal type of Morley rank $\omega$\texorpdfstring{\\}{ }in a partial differential field}
\author[P. KOWALSKI]{Piotr Kowalski}
\thanks{
 Supported by the Narodowe Centrum Nauki grant no. 2021/43/B/ST1/00405 and by the T\"{u}bitak grant no. 1001-124F359.}
\address{$^{\diamondsuit}$Instytut Matematyczny\\
Uniwersytet Wroc{\l}awski\\
Wroc{\l}aw\\
Poland}
\email{pkowa@math.uni.wroc.pl} \urladdr{http://www.math.uni.wroc.pl/\textasciitilde pkowa/ }

\begin{document}
\begin{abstract}
We consider the Kolchin-generic complete $1$-type of the inviscid Burgers equation
$Dx=x\partial x$ in a differentially closed field of characteristic zero
with two commuting derivations. We prove that this type has Lascar rank
$1$ and Morley rank $\omega$.
\end{abstract}
\maketitle

\section{Introduction and main result}

Model-theoretic ranks provide notions of dimension for definable sets
and complete types. In differentially closed fields, these ranks can
also be compared with differential-algebraic dimensions. The Lascar
rank of a complete type measures the complexity of its forking
extensions, whereas its Morley rank is the minimum of the Morley
ranks of its ``definable neighbourhoods''. These ranks can therefore
behave differently when a type avoids exceptional definable sets
that nevertheless contribute to the rank of every neighbourhood.

We are interested in the differential case, where rank discrepancies
were established by Hrushovski and Scanlon~\cite{HS}: Lascar rank and
Morley rank need not coincide, even on definable sets of finite
Morley rank in $\DCF_0$. With one derivation, however, finite
$U$-rank, finite Morley rank, and finite-dimensionality are
equivalent; see~\cite{Poizat}. Here a type is finite-dimensional
if the differential field generated by a realization has finite
transcendence degree over the differential base field.

For several commuting derivations, the situation is different.
McGrail~\cite{McGrail} claimed the inequalities
\[
 \omega^\tau\cdot a_\tau
 \leqslant U(q)
 \leqslant \RM(q)
 <\omega^\tau\cdot(a_\tau+1),
\]
where $\tau$ is the degree of the Kolchin polynomial of $q$ and
$a_\tau$ is its leading coefficient when written in the binomial
basis. These inequalities would imply the same equivalence of
the three finiteness conditions as in the ordinary differential
case. However, S\"uer~\cite{Suer} showed that the proposed lower
bound for $U$-rank fails. For the set of solutions of the \emph{heat equation}
\[
 H=\{x\mid \partial_1x=\partial_2^2x\}
\]
in $\DCF_{0,2}$, the Kolchin polynomial of a generic solution is
$2t+1$, but $U(H)=\omega$, rather than being at least
$\omega\cdot2$. This example disproves the lower bound, but does
not settle the asserted equivalence of the finiteness conditions.

The present paper was motivated by ongoing joint work with James
Freitag and Omar Le\'on S\'anchez on these questions; see also
the author's slides from a recent workshop in Chania  ~\cite{Chania}. In~\cite{LS}, Le\'on S\'anchez
asked whether the equation
\[
 \partial_1(x)=x^3+c,
 \qquad c\text{ differentially generic},
\]
defines a set of finite $U$-rank in $\DCF_{0,2}$. In our joint work,
we have shown that this set, and more general families, are
strongly minimal in $\DCF_{0,2}$, despite the infinite
transcendence degree of the differential field generated by a
generic solution. These examples therefore separate finiteness
of model-theoretic rank from finite-dimensionality, but their
$U$-rank and Morley rank are both $1$. This leaves the question
whether finite $U$-rank can coexist with infinite Morley rank. We answer this
question by considering the Kolchin-generic type of the inviscid
Burgers equation.

The \emph{inviscid Burgers equation} $u_t+u\,u_x=0$ belongs to the classical theory of nonlinear wave
propagation originating from Riemann's study of
finite-amplitude plane waves~\cite{Riemann1860}. Its modern name
reflects Burgers's investigation of the \emph{viscous equation}
$u_t+u\,u_x=\nu u_{xx}$ as a simplified model of turbulence,
notably in \cite{Burgers1948}, although this viscous
equation had already appeared in Bateman's work in
~\cite{Bateman1915}. Hopf and Cole developed
its reduction to the heat equation, now known as the Cole--Hopf
transformation, and investigated its solutions in ~\cite{Hopf1950,Cole1951}. The inviscid equation is
obtained by setting $\nu=0$.

We describe now our main result. Let us fix a sufficiently
saturated model
\[
 (\Ucal;D,\partial)\models\DCF_{0,2},
 \qquad D\partial=\partial D,
\]
and write
\[
 \Delta=\{D,\partial\},
 \qquad
 \C=\ker D\cap\ker\partial
\]
for the set of derivations and the field of total constants,
respectively. Choose $s,t\in\Ucal$ such that
\[
 Ds=0,\qquad \partial s=1,\qquad
 Dt=1,\qquad \partial t=0,
\]
and fix a small algebraically closed differential subfield
$M\subseteq\Ucal$ containing $s$ and $t$. Such $s$ and $t$ exist by
differential closedness. These parameters will be used in the
Morley-rank calculation.

Consider the solution set of the inviscid Burgers equation (up to reversing the time
variable)
\begin{equation}\label{burgers}
 X=\{x\in\Ucal: Dx=x\partial x\}.
 \notag
\end{equation}
We show in Lemma~\ref{irred} that $X$ is absolutely
Kolchin irreducible, so its Kolchin-generic type over $M$ is
well defined (see its description above the statement of Theorem \ref{independence}).

\begin{theorem}\label{main}
Let $p\in S_1(M)$ be the Kolchin-generic type of $X$. Equivalently,
$a\models p$ if and only if
\[
 Da=a\partial a
 \quad\text{and}\quad
 a,\partial a,\partial^2a,\ldots
 \text{ are algebraically independent over }M.
\]
Then
\[
 U(p)=1,\qquad \RM(p)=\omega, \qquad  \RM(X)=\omega,\qquad \MD(X)=1,\qquad U(X)=\omega.
\]
Moreover, the type $p$ is strictly disintegrated and hence $p$ is geometrically trivial.
\end{theorem}

The argument has two parts. The minimality argument rests on a
differential-algebraic independence statement: distinct solutions
that are individually $\partial$-differentially transcendental over
a base field have jointly algebraically independent
$\partial$-jets. Its proof uses standard properties of K\"ahler
differentials and the operators induced by $D$ and $\partial$.
Quantifier elimination then shows that any two ordered tuples
of distinct realizations of $p$, of the same length, have the
same type over $M$. Stability gives minimality.

For the Morley-rank calculation, every proper Kolchin-closed
subset of $X$ is finite-dimensional, while $X$ contains
finite-dimensional definable sets $Y_n$ of rank $n+1$ for every
$n\geqslant1$. These sets are given by the following polynomial
relations
\[
 s+tx=\sum_{i=0}^{n}c_i x^i,
 \qquad c_i\in\C,\quad c_n\neq0.
\]
The rank-one assertion concerns the \emph{complete type}.
The whole solution set satisfies
\[
 U(X)=\RM(X)=\omega,
\]
since the sets $Y_n$ also have $U$-rank $n+1$. In particular,
$p$ is nonisolated. The distinction between ranks of complete
types and ranks of definable sets leads to the question discussed
in Section~\ref{sec3}: whether finite $U$-rank and finite Morley rank are
equivalent for definable sets in $\DCF_{0,m}$.

The paper is organized as follows. In Section~\ref{sec2}, we recall the
necessary model-theoretic and differential-algebraic preliminaries,
describe the generic type of $X$, and prove the independence
theorem. In Section~\ref{sec3}, we compute the ranks and complete the proof
of Theorem~\ref{main}. We finish Section~\ref{sec3} with some further observations and
a question concerning definable sets.

\medskip
\noindent
\textbf{Acknowledgements and use of AI.}
The author thanks James Freitag and Omar Le\'on S\'anchez for their
collaboration on the work which motivated this paper as well as Sonat S\"{u}er for helpful discussions.

The author used ChatGPT (GPT-6 Astra in Pro mode) to assist in identifying examples, developing proofs, and drafting portions of the text. The author subsequently refined and streamlined the arguments and assumes full responsibility for the final content.

\section{Preliminaries and the independence theorem}\label{sec2}

In this section we provide model-theoretic and algebraic preliminaries. We start with two general stability results and then continue towards our partial differential context.

Let $T$ be a complete stable theory, $\Uu$ be a
sufficiently saturated monster model, $A\subset B\subset \Uu$ be small, and $p\in S_1(A)$
be a nonalgebraic stationary type. We write $a\ind_A B$ for nonforking independence over $A$ and set 
\[
P:=p(\Uu)\subset \Uu.
\]
We start with the standard minimality characterization (see e.g.~\cite[Remark 5.6]{Pillay}).
\begin{lemma}\label{indiscernible-minimal}
If any two finite ordered tuples (of the same length)
of distinct realizations of $p$ have the same type
over $A$, then $U(p)=1$.
\end{lemma}

\begin{proof}
Suppose a formula $\varphi(x;b)$ cuts $P$ into two
infinite pieces. Fix distinct $a_1,\ldots,a_n\in P$. For each
$I\subseteq\{1,\ldots,n\}$, choose distinct $a'_1,\ldots,a'_n\in P$
with $\varphi(a'_i;b)$ holding exactly when $i\in I$. By the hypothesis, there is $\sigma_I\in \operatorname{Aut}(\Uu/A)$
 sending $(a'_1,\ldots,a'_n)$ to
$(a_1,\ldots,a_n)$. Applying it to $b$, we obtain
\[
\varphi(a_i;\sigma_I(b))
\iff \varphi(a'_i;b)
\iff i\in I.
\]
Hence, for our fixed formula $\varphi$, we have
\[
\forall n\geqslant1\ \exists a_1,\ldots,a_n\
\forall I\subseteq\{1,\ldots,n\}\ \exists b_I\
\forall i\leqslant n\
\bigl(\varphi(a_i;b_I)\iff i\in I\bigr),
\]
which contradicts stability.

Thus every externally definable subset of $P$ is finite or cofinite.
Therefore
every nonalgebraic extension of $p$ must choose the cofinite side of
every formula. Such an extension is unique over every larger parameter
set. By stationarity it is the nonforking extension, proving $U(p)=1$.
\end{proof}

\begin{remark}\label{PDCT}
Assume in addition that $U(p)=1$. We collect here some well-known facts regarding geometric triviality.
\begin{enumerate}
\item[(P)]
\emph{Pure induced structure over $A$.}
For every $n\geqslant 1$, the restriction to $P^n$ of every $A$-definable
relation is definable without parameters in the language of equality.
Equivalently, every such restriction is a Boolean combination of
equalities $x_i=x_j$.

Another equivalent formulation is that for all $n$ all ordered $n$-tuples of pairwise distinct elements of
$P$ have the same type over $A$.

\item[(D)]
\emph{Strict disintegration over $A$.}
Every finite tuple $(a_1,\ldots,a_n)$ of pairwise distinct realizations of $p$ is
independent over $A$:
\[
a_i\ind_A a_1\ldots a_{i-1}
\qquad (1\leqslant i\leqslant n).
\]

\item[(C)]
\emph{The algebraic-closure condition.} For every $B\subseteq P$, we have the following
\[
\operatorname{acl}(A\cup B)\cap P=B.
\]
By finite character of algebraic closure, it suffices to require this
for finite $B$.

\item[(T)] For every $B\subseteq P$, let us write
\[
\operatorname{cl}_p(B):=\operatorname{acl}(A\cup B)\cap P
\qquad(B\subseteq P).
\]
The type $p$ is \emph{geometrically trivial}, or \emph{disintegrated},
over $A$ if for every $B\subseteq P$, we have
\[
\operatorname{cl}_p(B)
=
\bigcup_{b\in B}\operatorname{cl}_p(\{b\}).
\]
\end{enumerate}
Then, we have the following:
\[
\boxed{\textnormal{(P)}
\ \Longleftrightarrow\
\textnormal{(D)}
\ \Longleftrightarrow\
\textnormal{(C)}\ \Longrightarrow\
\textnormal{(T)}.}
\]
Indeed, stationarity implies that all independent ordered
$n$-tuples of realizations of $p$ have the same type over $A$,
so \textnormal{(D)} implies \textnormal{(P)}. Conversely,
comparing any tuple of distinct realizations with a Morley
sequence in $p$ gives \textnormal{(P)} implies
\textnormal{(D)}. Since $U(p)=1$, an extension of $p$ forks
over $A$ if and only if it is algebraic. This gives the
equivalence of \textnormal{(D)} and \textnormal{(C)}.
Finally, \textnormal{(C)} immediately implies
\textnormal{(T)}.
We can describe more precisely:
\[
p\text{ is strictly disintegrated}
\quad\Longleftrightarrow\quad
p\text{ is disintegrated $\ $and $\ \left(\forall a\in P\right)\left(\operatorname{cl}_p(\{a\})=\{a\}\right)$}
.
\]
\end{remark}

We move now to our case of interest: the theory $\DCF_{0,2}$. We have the set-up from the introduction, that is $(\Ucal;D,\partial)$ is a monster model and $X$ is defined by the inviscid Burgers equation.
We use the standard facts that $\DCF_{0,2}$ is $\omega$-stable and has
quantifier elimination, that algebraic closure is field-theoretic
algebraic closure of the generated differential field, and that types
over algebraically closed differential fields are stationary.
Definable sets are Boolean combinations of Kolchin-closed sets, and
complete types are determined by their differential polynomial
equations; see McGrail~\cite{McGrail} and the summary in
Le\'on S\'anchez~\cite{LS}.

Let $\Delta=\{D,\partial\}$ and $K\subseteq \Uu$ be a $\Delta$-differential subfield. For any $a\in \Uu$, we denote by $K\langle a\rangle_\Delta$ the smallest $\Delta$-differential subfield of $\Uu$ containing $K$ and $a$. A tuple $a$ is \emph{finite-dimensional over $K$} if
$\trdeg(K\langle a\rangle_\Delta/K)<\infty$.
We also use the standard  rank bound stated below and include its proof for completeness.

\begin{lemma}\label{finite-dimensional-bound}
Let $K\subset \Uu$ be a small algebraically closed $\Delta$-differential field,
and let $Z$ be a $K$-definable set. Then, for any $d\geqslant0$, we have the following
\[
\left(\forall a\in Z \ \trdeg(K\langle a\rangle_\Delta/K)\leqslant d\right)\qquad \Longrightarrow\qquad
\qquad \RM(Z)\leqslant d.
\]
\end{lemma}
\begin{proof}
We may assume that $Z$ is nonempty. Choose $a\in Z$
such that
\[
\RM(a/K)=\RM(Z),
\]
and put
\[
e=\trdeg(K\langle a\rangle_\Delta/K)\leqslant d.
\]
If $e=0$, then $a$ is algebraic over $K$, so
$\RM(Z)=0$. Otherwise, the Kolchin polynomial of
$q=\tp(a/K)$ is the constant polynomial $e$.
Thus $\tau(q)=0$ and $a_0(q)=e$.

The upper-bound argument in~\cite[Lemma~5.2.6 and
Proposition~4.2.9]{McGrail}, applied to the constant
Kolchin polynomial $e$, gives
\[
\RM(q)<\omega^0(e+1)=e+1.
\]
Hence $\RM(q)\leqslant e$.
\end{proof}
The next result gives Kolchin irreducibility.
\begin{lemma}\label{irred}
For any $\Delta$-differential subfield $K$ of $\Uu$, let $K\{Y\}_{\Delta}$ denote the $\Delta$-differential ring of $\Delta$-differential polynomials over $K$. Let $I$ be the $\Delta$-differential ideal of $K\{Y\}_{\Delta}$ which is $\Delta$-differentially generated by $DY-Y\partial Y$. Then we have the following $\Delta$-differential isomorphism over $K$
\begin{equation}\label{coordinate-ring}
 K\{Y\}_\Delta/I
 \cong K[y_0,y_1,y_2,\ldots],
 \notag
\end{equation}
where $y_0,y_1,y_2,\ldots$ are algebraically independent over $K$ and
\begin{equation}\label{jet-actions}
 \partial y_j=y_{j+1},\qquad
 Dy_j=\partial^j(y_0y_1).
\notag
\end{equation}
In particular, $X$ is absolutely Kolchin irreducible.
\end{lemma}
\begin{proof}
We put $R:=K[y_0,y_1,y_2,\ldots]$, where the $y_0,y_1,y_2,\ldots$ are algebraically
independent over $K$. Extend the derivations of $K$ to $R$ by
\[
\partial y_j=y_{j+1},
\qquad
Dy_j=\partial^j(y_0y_1)
=\sum_{k=0}^{j}\binom{j}{k}y_ky_{j-k+1}.
\]
These derivations commute on $K$ and on every generator, since
\[
D(\partial y_j)=Dy_{j+1}
=\partial^{j+1}(y_0y_1)=\partial(Dy_j).
\]
Their commutator is a derivation, so they commute on all of $R$.
Moreover, we have $Dy_0=y_0\partial y_0$. Hence the $\Delta$-differential evaluation map 
\[
\operatorname{ev}_{y_0}:K\{Y\}_\Delta\longrightarrow R
\]
induces a $\Delta$-differential $K$-algebra homomorphism
\[
\Phi:K\{Y\}_{\Delta}/I\longrightarrow R.
\]
It is enough to show that $\Phi$ is a bijection. We define a $K$-algebra homomorphism
\[
\Psi:R\longrightarrow K\{Y\}_{\Delta}/I,
\qquad
y_j\longmapsto\overline{\partial^jY},
\]
where for $f\in K\{Y\}_{\Delta}$, we set $\bar{f}:=f+I$. The relation defining $I$ gives
\[
D\overline{\partial^jY}
=\overline{\partial^j(Y\partial Y)}
=\sum_{k=0}^{j}\binom{j}{k}
  \overline{\partial^kY}\,
  \overline{\partial^{j-k+1}Y}.
\]
Thus the image of $\Psi$ is closed under both derivations.
Since this image contains $K$ and $\overline Y$, the map $\Psi$ is onto.
Since we also have $\Phi\circ\Psi=\operatorname{id}_R$, both maps $\Phi$ and $\Psi$ are bijections.

For the ``In particular'' part, we have shown that the ring $R$ is an integral domain, so $I$ is a prime differential
ideal. By the differential Nullstellensatz
\cite[Chapter~IV, \S3, Theorem~2, p.~147]{Kolchin},
 $X$ is Kolchin
irreducible. The same construction works after replacing $K$ by
any differential field extension; therefore $X$ is absolutely
Kolchin irreducible.
\end{proof}
We need the following.
\begin{proposition}
Let $(\mathcal U,\Delta)\models\mathrm{DCF}_{0,m}$, let
$K\subseteq\mathcal U$ be a differential subfield, and let
$V\subseteq\mathcal U^n$ be a nonempty Kolchin-closed set defined
over $K$ and irreducible over $K$. Then Kolchin genericity in $V$
over $K$ determines a unique complete type over $K$.
\end{proposition}
This proposition follows from a more general result below (see the paragraph immediately preceding \cite[Proposition 4.7]{HK}).
\begin{remark}
Let 
$\mathfrak C\models T$ be a monster model of a Noetherian theory and let
$A\subseteq\mathfrak C$ be small. Thus the $A$-closed sets form
a Noetherian topology, and the $A$-definable sets are precisely
the Boolean combinations of $A$-closed sets.

If $V\subseteq\mathfrak C^n$ is nonempty, $A$-closed, and
$A$-irreducible, then the partial type
\[
 \Sigma_V(x)
 =
 \{x\in V\}
 \,\cup\,
 \{x\notin W:
     W\subsetneq V\text{ is }A\text{-closed}\}
\]
has a unique complete extension $p_V\in S_n(A)$. Moreover, the assignment $V\mapsto p_V$ is a bijection between
nonempty $A$-closed $A$-irreducible subsets of $\mathfrak C^n$
and $S_n(A)$.
\end{remark}
By Lemma~\ref{irred}, the
Kolchin-generic type $p(x)$ of $X$ over $M$ is determined by:
\[
Dx=x\partial x\qquad \text{and}\qquad \text{$x$ is $\partial$-differentially transcendental over $M$}.
\]
We prove now the main result of this section.
\begin{theorem}\label{independence}
Let $K\subseteq \Uu$ be a $\Delta$-differential subfield and $a_1,\ldots,a_n\in \Uu$ be pairwise distinct such that for each $1\leqslant i\leqslant n$ we have:
\begin{itemize}
  \item $ Da_i=a_i\partial a_i$,
  \item $a_i$ is $\partial$-differentially transcendental over $K$.
\end{itemize}
Then the sequence
\[
 (\partial^j a_i:1\leqslant i\leqslant n,\ j\geqslant0)
\]
is algebraically independent over $K$.
\end{theorem}

\begin{proof}
We will use the module of K\"{a}hler differentials and its standard properties (see e.g. \cite{Stacks}). Most importantly, we use the fact that, for a field extension $K\subseteq F$ in characteristic zero, a family of
elements of $F$ is algebraically independent over $K$ if
and only if its differentials are $F$-linearly independent
in $\Omega_{F/K}$. Put
\[
 F:=K(\partial^j a_i:1\leqslant i\leqslant n,\ j\geqslant0),
 \qquad \Omega:=\Omega_{F/K}.
\]
Then $F$ is stable under both
derivations and they induce the following commuting Lie derivatives on $\Omega$
\[
 \delta_\Omega(f\,\dd g)
 =\delta(f)\,\dd g+f\,\dd(\delta g),
 \qquad \delta\in\{D,\partial\}.
\]
These operators are well defined because both derivations preserve $K$. Let us write
\[
 e_{i,j}=\dd(\partial^j a_i),\qquad
 E_i^{\leqslant m}=e_{i,0}F+\ldots+e_{i,m}F,\qquad
 E_i=\bigcup_{m\geqslant0}E_i^{\leqslant m}.
\]
For each fixed $i$, the family $(e_{i,j})_{j\geqslant0}$ is $F$-linearly
independent by our transcendence assumption. Thus every nonzero $v\in E_i$ has a unique expression
\[
 v=c_m e_{i,m}+\cdots+c_0e_{i,0},\qquad c_m\neq0.
\]
Let us define $\ord_i(v)=m$ and write
\[
v=c_me_{i,m}+\text{l.o.t.},
\]
where l.o.t. stands for lower-order terms.
We have the following.
\begin{align}
D_\Omega(e_{i,0})
  &= \dd\bigl(D(a_i)\bigr) \notag\\
  &= \dd\bigl(a_i\,\partial(a_i)\bigr) \notag\\
  &= a_i\,\dd\bigl(\partial(a_i)\bigr)
     + \partial(a_i)\,\dd(a_i) \notag\\
  &= a_i e_{i,1} + \partial(a_i)e_{i,0} \notag.
\end{align}
Applying $\partial_\Omega^j$, using commutativity and the Leibniz rule,
we obtain
\begin{equation}\label{leading}
 \partial_\Omega e_{i,j}=e_{i,j+1},\qquad
 D_\Omega e_{i,j}-a_i e_{i,j+1}\in E_i^{\leqslant j}
 \quad(j\geqslant0).
\end{equation}
{\bf Claim}
\\
The sum of the spaces $E_i$ is direct.
\begin{proof}[Proof of Claim]
Suppose not and choose a
relation involving the fewest possible spaces, and renumber:
\begin{equation}\label{minimal}
 v_1+\cdots+v_r=0,\qquad
 0\neq v_i\in E_i,\qquad r\geqslant2.
\end{equation}
The key point is that $D_\Omega$ has leading coefficient
$a_i$ on the $i$th family. The distinctness of the $a_i$
therefore allows us to cancel the leading term on one
family while retaining it on all the others. Set
\[
 T:=D_\Omega-a_1\partial_\Omega, \qquad T:\Omega\longrightarrow \Omega.
\]
Then $T$ need not be $F$-linear, but it is additive and satisfies
\begin{equation}\label{semilinear}
 T(fv)=(Df-a_1\partial f)v+fT(v),\qquad f\in F, v\in \Omega.
\end{equation}
Equations \eqref{leading} and \eqref{semilinear} imply
\begin{equation}\label{preserve}
 T\left(E_1^{\leqslant m}\right)\subseteq E_1^{\leqslant m}
 \quad\text{for every }m\geqslant0.
\end{equation}
On the other hand, for $i\neq1$ and
$v=c_m e_{i,m}+\text{lower-order terms}$ with $c_m\neq0$, they give
\[
 Tv=c_m(a_i-a_1)e_{i,m+1}+\text{lower-order terms}.
\]
Since $a_i\neq a_1$, it follows that
\begin{equation}\label{increase}
 \ord_i(Tv)=\ord_i(v)+1
 \qquad(i\neq1,\ 0\neq v\in E_i).
\end{equation}
Let $m=\ord_1(v_1)$. By \eqref{preserve}, all the vectors
\[
 v_1,Tv_1,\ldots,T^{m+1}v_1
\]
belong to the $(m+1)$-dimensional space $E_1^{\leqslant m}$. Hence there
exist $c_0,\ldots,c_q\in F$, with $1\leqslant q\leqslant m+1$ and $c_q\neq0$,
such that
\begin{equation}\label{annihilate}
 \sum_{k=0}^q c_kT^kv_1=0.
\end{equation}
For $i=2,\ldots,r$, put
\[
 w_i=\sum_{k=0}^q c_kT^kv_i\in E_i.
\]
By \eqref{increase}, the vectors $T^kv_i$ have strictly increasing
orders. Therefore the summand $c_qT^qv_i$ has strictly higher order
than every other nonzero summand, and $w_i\neq0$.

Apply $T^k$ to \eqref{minimal}, multiply by $c_k$, and sum over $k$.
Additivity of $T$ and \eqref{annihilate} yield
\[
 w_2+\cdots+w_r=0,\qquad 0\neq w_i\in E_i.
\]
This contradicts the minimality of $r$.
\end{proof}
By the claim, the sum of the $E_i$ is direct. Since each individual
family $(e_{i,j})_{j\geqslant0}$ is linearly independent, all the
$e_{i,j}$ are jointly $F$-linearly independent, which gives the desired algebraic independence.
\end{proof}

\section{Rank computations}\label{sec3}
Throughout this section, $p\in S_1(M)$ denotes the
Kolchin-generic type from Theorem~\ref{main}. Having the preliminary results from Section 2, we put them together in this section to obtain a proof of Theorem~\ref{main}. We start from the U-rank computations.
\begin{proposition}\label{minimality}
We have $U(p)=1$.
\end{proposition}

\begin{proof}
We will check the condition from Lemma~\ref{indiscernible-minimal}. Let us take
\[
a_1,b_1,\ldots,a_n,b_n\models p,\qquad |\{a_1,\ldots,a_n\}|=n=|\{b_1,\ldots,b_n\}|.
\]
By Theorem~\ref{independence}, we get
\[
\operatorname{qftp}(a_1,\ldots,a_n/M)=\operatorname{qftp}(b_1,\ldots,b_n/M).
\]
By quantifier elimination, the complete types coincide as well, so the result holds by Lemma~\ref{indiscernible-minimal}.
\end{proof}
We can show now geometric triviality.
\begin{corollary}\label{triviality}
For every $A\subseteq p(\Ucal)$, we have
\[
 \acl(M\cup A)\cap p(\Ucal)=A.
\]
In particular, the geometry of $p$ is strictly disintegrated (hence trivial, see Remark \ref{PDCT}).
\end{corollary}
\begin{proof}
The inclusion $A\subseteq\acl(M\cup A)\cap p(\Ucal)$ is clear.
Conversely, let $b\in\acl(M\cup A)\cap p(\Ucal)$.
By finite character, choose distinct $a_1,\ldots,a_n\in A$
such that $b\in\acl(M\cup \{a_1,\ldots, a_n\})$.
Suppose that $b\notin A$. By Theorem~\ref{independence},
the $\partial$-jets of $a_1,\ldots,a_n,b$ are jointly
algebraically independent over $M$. In particular, $b$
is transcendental over
\[
F:=M(\partial^j a_i:1\leqslant i\leqslant n,\ j\geqslant0)
  =M\langle a_1,\ldots,a_n\rangle_\Delta.
\]
This contradicts
\[
b\in\acl(M\cup \{a_1,\ldots, a_n\})=F^{\mathrm{alg}}\cap\Ucal.
\]
Thus $b\in A$.
\end{proof}

We move now to Morley rank computations.
\begin{lemma}\label{proper-finite}
Every proper Kolchin-closed subset of $X$, over any differential
parameter field, has finite Morley rank.
\end{lemma}

\begin{proof}
Let $Z\subsetneq X$ be Kolchin closed, and choose a small
algebraically closed differential field $K$ containing
its parameters of definition. We may assume that $Z$ is nonempty.

By Lemma~\ref{irred}, there is $r\geqslant 0$ and $Q\in K[X_0,\ldots,X_r]\setminus \{0\}$ such that
\[
\forall x\in Z \ Q(x,\partial x,\ldots,\partial^r x)=0\qquad \text{and} \qquad \exists x\in X\ Q(x,\partial x,\ldots,\partial^r x)\neq 0.
\]
By Lemma~\ref{finite-dimensional-bound}, it is enough to show the following
\begin{equation}\label{finite-transcendence-bound}
\forall a\in Z\qquad \trdeg(K\langle a\rangle_\Delta/K)\leqslant r.
\tag{$*$}
\end{equation}

Fix $a\in Z$ and write $a_j=\partial^j a$.
Let $h\leqslant r$ be least such that $a_0,\ldots,a_h$
are algebraically dependent over $K$. Then $a_h$ is
algebraic over $K(a_0,\ldots,a_{h-1})$, so $a_{h+1}\in K(a_0,\ldots,a_h)$.
Hence $K(a_0,\ldots,a_h)$ is closed under $\partial$.
Moreover, the Burgers equation gives
\[
Da_j=\partial^j(a_0a_1)\in K(a_0,\ldots,a_h),
\]
so $K(a_0,\ldots,a_h)$ is also closed under $D$. Therefore, we obtain
\[
K\langle a\rangle_\Delta=K(a_0,\ldots,a_h),
\qquad
\trdeg(K(a_0,\ldots,a_h)/K)=h\leqslant r,
\]
which gives (\ref{finite-transcendence-bound}).
\end{proof}

\begin{proposition}\label{upper-bound}
Any definable subset of $X$ either has
finite Morley rank or has complement of finite Morley rank. In particular, $\RM(X)\leqslant\omega$.
\end{proposition}

\begin{proof}
Let $V$ be a definable subset of $X$. By Lemma \ref{irred}, we have $\overline{V}=X$ or $\overline{X\setminus V}=X$, where $\overline{T}$ is the Kolchin closure of $T$. Therefore, by quantifier elimination, we get that $V$ has non-empty Kolchin interior or $X\setminus V$ has non-empty Kolchin interior. Assume that $V$ has non-empty Kolchin interior. Then, $X\setminus V$ is contained in a proper Kolchin closed subset of $X$, so $\RM(X\setminus V)<\omega$ by Lemma \ref{proper-finite}.

For the ``In particular'' part, if $\RM(X)\geqslant\omega+1$, then there are two disjoint
definable subsets of $X$ of Morley rank at least $\omega$.
By the preceding argument, the complement of the first
has finite Morley rank, contradicting the rank of the
second. Hence $\RM(X)\leqslant\omega$.
\end{proof}
For $n\geqslant1$, let
\[
 T_n=\{(c_0,\ldots,c_n)\in\C^{n+1}\mid c_n\neq0\}
\]
and define (recall the fixed $s,t\in \Uu$ from the Introduction):
\begin{equation}\label{Yn}
 Y_n=\left\{x\mid \exists\bar c\in T_n,\quad
          s+tx=\sum_{i=0}^{n}c_i x^i\right\}.
          \notag
\end{equation}
\begin{proposition}\label{exceptional-rank}
For every $n\geqslant1$, we have
\[
Y_n\subseteq X\qquad \text{and}\qquad U(Y_n)=\RM(Y_n)=n+1.
\]
\end{proposition}
\begin{proof}
Let $x\in Y_n$, and choose $\bar c=(c_0,\ldots,c_n)\in T_n$ such that
\[
s+tx=P(x),
\qquad
P(T)=\sum_{i=0}^n c_iT^i\in \C[T].
\]
Differentiation gives
\[
(P'(x)-t)\partial x=1,
\qquad
(P'(x)-t)Dx=x.
\]
Thus we obtain $\partial (x)\neq0$ and $D(x)=x\partial (x)$, so $Y_n\subseteq X$.

Since $\C$ is algebraically closed and $x\notin\C$,
the element $x$ is transcendental over $\C$.
Consequently,
\[
\forall \bar c,\bar d\in T_n\qquad \qquad \left(
\sum_{i=0}^n c_i x^i
=s+tx
=\sum_{i=0}^n d_i x^i
\qquad \Longrightarrow\qquad \bar c=\bar d\right).
\]
Hence the relation defining
$Y_n$ determines a definable function
\[
f:Y_n\longrightarrow T_n,\qquad x\longmapsto\bar c.
\]

For every $\bar c\in T_n$, the polynomial
$P(T)-tT-s$ has degree $n$. For $n=1$, this uses
$c_1-t\neq0$, which follows from $Dt=1$.
Since $\Ucal$ is algebraically closed, this polynomial
has a root. Thus $f$ is surjective and has finite fibres.

The field $\C$ is stably embedded with its induced
structure that of a pure algebraically closed field.
Therefore
\[
\RM(T_n)=U(T_n)=n+1.
\]
Since $f$ is a finite-to-one definable surjection,
\[
\RM(Y_n)=U(Y_n)=n+1,
\]
which we needed to show.
\end{proof}

\begin{proof}[Completion of the proof of Theorem~\ref{main}]
Propositions~\ref{upper-bound} and~\ref{exceptional-rank} imply
$\RM(X)=\omega$. The finite-rank-complement assertion in
Proposition~\ref{upper-bound} also gives $\MD(X)=1$.

If $\varphi(x)\in p$, then $X\cap\varphi(\Ucal)$ contains a
Kolchin-open subset of $X$ over $M$. Its complement in $X$ is contained
in a proper Kolchin-closed subset and has finite Morley rank.
Therefore $\RM(X\cap\varphi(\Ucal))=\omega$. Since $X\in p$,
this proves $\RM(p)=\omega$.

Since for each $n\geqslant1$, Proposition \ref{exceptional-rank} gives:
\[
n+1=U(Y_n)\leqslant U(X)\leqslant\RM(X)=\omega,
\]
we also have $U(X)=\omega$. The remaining assertions are Proposition~\ref{minimality} and
Corollary~\ref{triviality}.
\end{proof}

\begin{remark}
We collect here some final observations.
\begin{enumerate}
  \item For $a\models p$, we have the following
\[
 M\bigl(D^i\partial^j a:i+j\leqslant r\bigr)
   =M(a,\partial a,\ldots,\partial^r a).
\]
Consequently its Kolchin polynomial is $r+1$, the differential type is
one, and the typical differential dimension is one.

  \item Each $Y_n$ from the proof of Proposition \ref{exceptional-rank} is disjoint from $p(\Ucal)$. Indeed, $Y_n$ is
$M$-definable, so if it contained a realization of $p$,
the formula defining $Y_n$ would belong to $p$. This would
imply
\[
\RM(p)\leqslant\RM(Y_n)=n+1,
\]
contrary to $\RM(p)=\omega$.

  \item The type $p$ is nonisolated. Indeed, if it were isolated, its isolating
set would equal $p(\Ucal)$ and would be strongly minimal by
Proposition~\ref{minimality}. Its Morley rank would then be one,
contrary to Theorem~\ref{main}.
\end{enumerate}

\end{remark}
I have not found an example of a definable set $V$ in
$\mathrm{DCF}_{0,2}$ such that
\[
U(V)=1
\qquad\text{and}\qquad
\mathrm{RM}(V)=\omega.
\]
Actually, one can show that for our set $X$ from the introduction and any definable $V\subseteq X$, we have
\[
U(V)<\omega\qquad\iff\qquad\RM(V)<\omega.
\]
This leads to the following.
\begin{question}
Does every definable set of finite
$U$-rank in $\DCF_{0,m}$ have finite Morley rank?
\end{question}

\end{document}